\documentclass[12pt]{elsarticle-ppn}
\usepackage{amsmath,amsfonts,amsthm,amssymb,amsxtra,hyperref,color,graphicx,setspace,xspace}
\usepackage{hyperref}

\newtheorem{thm}{Theorem}
\newtheorem{cor}[thm]{Corollary}
\newtheorem{lem}[thm]{Lemma}
\newtheorem{prop}[thm]{Proposition}
\newtheorem{remark}[thm]{Remark}

\newcommand{\D}{\mathcal D^{1,2}(\R^d)}

\renewcommand{\L}{\mathrm L}
\newcommand{\N}{\mathbb N}
\newcommand{\R}{\mathbb R}
\renewcommand{\S}{\mathbb S}
\newcommand{\be}[1]{\begin{equation}\label{#1}}
\newcommand{\ee}{\end{equation}}
\renewcommand{\(}{\left(}
\renewcommand{\)}{\right)}

\newcommand{\ird}[1]{\int_{\R^d}{#1}\;dx}

\newcommand{\irt}[1]{\int_{\R}{#1}\;dt}
\newcommand{\irdeux}[1]{\int_{\R^2}{#1}\;dx}
\newcommand{\irdmu}[1]{\int_{\R^2}{#1}\;d\mu}
\newcommand{\irdmua}[1]{\int_{\R^2}{#1}\;d\mu_\alpha}
\newcommand{\nrm}[2]{\|{#1}\|_{\L^{#2}(\R^d)}}
\newcommand{\nrmdeux}[2]{\|{#1}\|_{\L^{#2}(\R^2)}}
\newcommand{\scalar}[2]{\langle{#1},{#2}\rangle}

\newcommand{\Sd}{\mathsf{S}_d}

\newcommand{\isd}[1]{\int_{\S^d}{#1}\;d\sigma_d}

\begin{document}

\begin{frontmatter}



\title{Sobolev and Hardy-Littlewood-Sobolev inequalities}
\author{Jean~Dolbeault}
\ead{dolbeaul@ceremade.dauphine.fr}
\ead[url]{http://www.ceremade.dauphine.fr/$\sim\kern1.5pt$dolbeaul/}
\author{Gaspard~Jankowiak}
\ead{jankowiak@ceremade.dauphine.fr}
\ead[url]{http://gjankowiak.github.io/}
\address{Ceremade, Universit\'e Paris-Dauphine, Place de Lattre de Tassigny, 75775 Paris C\'edex~16, France.}


\begin{abstract} This paper is devoted to improvements of Sobolev and Onofri inequalities. The additional terms involve the dual counterparts, \emph{i.e.}~Hardy-Littlewood-Sobolev type inequalities. The Onofri inequality is achieved as a limit case of Sobolev type inequalities. Then we focus our attention on the constants in our improved Sobolev inequalities, that can be estimated by completion of the square methods. Our estimates rely on nonlinear flows and spectral problems based on a linearization around optimal Aubin-Talenti functions.
\end{abstract}


\end{frontmatter}

\section{Introduction}\label{Sec:Intro}

E.~Carlen, J.A.~Carrillo and M.~Loss noticed in \cite{CCL} that Hardy-Littlewood-Sobolev inequalities in dimension $d\ge3$ can be deduced from some special Gagliardo-Nirenberg inequalities using a fast diffusion equation. Sobolev's inequalities and Hardy-Littlewood-Sobolev inequalities are dual. A fundamental reference for this issue is E.H.~Lieb's paper \cite{MR717827}. This duality has also been investigated using a fast diffusion flow in \cite{Dolbeault2011}. Although \cite{CCL} has motivated~\cite{Dolbeault2011}, the two approaches are so far unrelated. Actually \cite{Dolbeault2011} is closely connected with the approach by Legendre's duality developed in \cite{MR717827}. We shall take advantage of this fact in the present paper and also use of the flow introduced in \cite{Dolbeault2011}.

\medskip For any $d\ge3$, the space $\mathcal D^{1,2}(\R^d)$ is defined as the completion of smooth solutions with compact support w.r.t.~the norm
\[
w\mapsto\|w\|:=\(\nrm{\nabla w}2^2+\nrm w{2^*}^2\)^{1/2}\,,
\]
where $2^*:=\frac{2\,d}{d-2}$. The Sobolev inequality in $\R^d$ is
\be{Ineq:Sobolev}
\mathsf{S}_d\,\nrm{\nabla u}2^2-\nrm u{2^*}^2 \ge0\quad\forall\,u\in\D\,,
\ee
where the best constant, or Aubin-Talenti constant, is given by
\[
\mathsf{S}_d=\frac{1}{\pi\,d\,(d-2)}\,\Big(\tfrac{\Gamma(d)}{\Gamma\(\frac{d}2\)}\Big)^\frac{2}d\
\]
(see~\ref{Sec:UsefulFormulae} for details). The optimal Hardy-Littlewood-Sobolev inequality
\be{Ineq:HLS}
\mathsf{S}_d\,\nrm{v}{\frac{2\,d}{d+2}}^2-\ird{v\,(-\Delta)^{-1}\,v}\ge0\quad\forall\;v\in\L^\frac{2\,d}{d+2}(\R^d)
\ee
involves the same best constant $\mathsf{S}_d$, as a result of the duality method of~\cite{MR717827}.  When $d\ge5$, using a well chosen flow, it has been established in~\cite{Dolbeault2011} that the l.h.s.~in~\eqref{Ineq:Sobolev} is actually bounded from below by the l.h.s.~in~\eqref{Ineq:HLS}, multiplied by some positive proportionality constant. In our first result, we will remove the technical restriction $d\ge5$ and cover all dimensions $d\ge3$. An elementary use of the duality method -- in fact a simple completion of the square method~-- provides a simple upper bound on the optimal proportionality constant in any dimension.
\begin{thm}\label{Thm:SquareSobolev} For any $d\ge3$, if $q=\frac{d+2}{d-2}$ the inequality
\begin{multline}\label{S-HLS}
\mathsf{S}_d\,\nrm{u^q}{\frac{2\,d}{d+2}}^2-\ird{u^q\,(-\Delta)^{-1}\,u^q}\\
\le\mathsf{C}_d\,\nrm u{2^*}^\frac{8}{d-2}\left[\mathsf{S}_d\,\nrm{\nabla u}2^2-\nrm u{2^*}^2\right]
\end{multline}
holds for any $u\in\D$ where the optimal proportionality constant $\mathsf{C}_d$ is such that
\[
\frac{d}{d+4}\,\mathsf{S}_d\le\mathsf{C}_d<\mathsf S_d\,.
\]\end{thm}
Inequality~\eqref{S-HLS} is obtained with $\mathsf{C}_d$ replaced by $\mathsf{S}_d$ by expanding a well chosen square in Section~\ref{Sec:Square}. The lower bound on $\mathsf C_d$ follows from an expansion of both sides of the inequality around the Aubin-Talenti functions, which are optimal for Sobolev and Hardy-Littlewood-Sobolev inequalities (see Section~\ref{Sec:Square} for more details), and spectral estimates that will be studied in Section~\ref{Sec:Linearization}: see Corollary~\ref{cor:linearization4}. The computation based on the flow as was done in \cite{Dolbeault2011} can be optimized to get an improved inequality compared to~\eqref{S-HLS}, far from the Aubin-Talenti functions: see Theorem~\ref{Thm:Improved} in Section~\ref{Sec:Improvements}. As a consequence, we also prove the strict inequality $\mathsf{C}_d<\mathsf{S}_d$.

\medskip In dimension \hbox{$d=2$}, consider the probability measure $d\mu$ defined by
\[
d\mu(x):=\mu(x)\,dx\quad\mbox{with}\quad \mu(x):=\frac{1}{\pi\,(1+|x|^2)^2}\quad\forall\;x\in\R^2.
\]
The Euclidean version of Onofri's inequality~\cite{MR677001}
\be{Ineq:Onofri}
\frac{1}{16\,\pi}\irdeux{|\nabla f|^2}-\log\(\irdmu{e^{\,f}}\)+\irdmu f\ge0\quad\forall\;f\in\mathcal D(\R^2)
\ee
plays the role of Sobolev's inequality in higher dimensions. Here the inequality is written for smooth and compactly supported functions in $\mathcal D(\R^2)$, but can be extended to the appropriate Orlicz space which corresponds to functions such that both sides of the inequality are finite.

This inequality is dual of the logarithmic Hardy-Littlewood-Sobolev inequality that can be written as follows: for any $g\in\L^1_+(\R^2)$ with $M=\irdeux g$, such that $g\,\log g$, $(1+\log|x|^2)\,g\in\L^1(\R^2)$, we have
\be{Ineq:logHLS}
\irdeux{g\,\log\(\frac{g}M\)}-\frac{4\,\pi}M\irdeux{g\,(-\Delta)^{-1}\,g}+M\,(1+\log\pi)\ge 0
\ee
with
\[
\irdeux{g\,(-\Delta)^{-1}\,g}=-\frac1{2\,\pi}\int_{\R^2\times\R^2}g(x)\,g(y)\,\log|x-y|\;dx\,dy\,.
\]
Then, in dimension $d=2$, we have an analogue of Theorem~\ref{Thm:SquareSobolev}, which goes as follows.
\begin{thm}\label{Thm:Onofri} The inequality
\begin{multline}\label{S-HLS-d=2}
\irdeux{g\,\log\(\frac{g}M\)}-\frac{4\,\pi}M\irdeux{g\,(-\Delta)^{-1}\,g}+M\,(1+\log\pi)\\
\le M\left[\frac1{16\,\pi}\,\nrmdeux{\nabla f}2^2+\irdmu f-\log M\right]
\end{multline}
holds for any function $f\in\mathcal D(\R^2)$ such that $M=\irdmu{e^{\,f}}$ and $g=e^f\,\mu$.\end{thm}
Using for instance \cite{Almgren-Lieb89} or \cite[Lemma~2]{MR1143664} (also see \cite[chapter 3--4]{lieb2001analysis}), it is known that optimality is achieved in~\eqref{Ineq:Sobolev}, \eqref{Ineq:HLS}, \eqref{Ineq:Onofri} or \eqref{Ineq:logHLS} when the problem is reduced to radially symmetric functions. However, no such result applies when considering a difference of the terms in two such inequalities, like in~\eqref{S-HLS} or~\eqref{S-HLS-d=2}. Optimality therefore requires a special treatment. In Section~\ref{Sec:Square}, we shall use the \emph{completion of the square method} to establish the inequalities (without optimality) under an assumption of radial symmetry in case of Theorem~\ref{Thm:Onofri}. For radial functions, Theorem~\ref{Thm:SquareSobolev} can indeed be written with $d>2$ considered as a real parameter and Theorem~\ref{Thm:Onofri} corresponds, in this setting, to the limit case as $d\to2_+$. To handle the general case (without radial symmetry assumption), a more general setting is required. In Section~\ref{Sec:CKN}, we extend the results established for Sobolev inequalities to weighted spaces and obtain an improved version of the Caffarelli-Kohn-Nirenberg inequalities (see Theorem~\ref{Thm:CKNsquare}). Playing with weights is equivalent to varying $d$ or taking limits with respect to~$d$, except that no symmetry assumption is required. This allows to complete the proof of Theorem~\ref{Thm:Onofri}.

Technical results regarding the computation of the constants, a weighted Poincar\'e inequality and the stereographic projection, the extension of the flow method of \cite{Dolbeault2011} to the case of the dimensions $d=3$ and $d=4$, and symmetry results for Caffarelli-Kohn-Nirenberg inequalities have been collected in various appendices.

At this point, we emphasize that Theorems~\ref{Thm:CKNsquare} and~\ref{Thm:Onofrisquare}, which are used as intermediate steps in the proof of Theorem~\ref{Thm:Onofri} are slightly more general than, respectively, Theorems~\ref{Thm:SquareSobolev} and~\ref{Thm:Onofri}, except for the issue of the optimal value of the proportionality constant, which has not been studied. It is likely that the method used for Sobolev's inequality can be adapted, but since weights break the translation invariance, some care should be given to this question, which is of independent interest and known to raise a number of difficulties of its own (see for instance \cite{Oslo}). The question of a lower estimate of the proportionality constant in \eqref{S-HLS-d=2} in connection with a larger family of Onofri type inequalities is currently being studied, see~\cite{Jankowiak2014}.

\medskip Let us conclude this introduction by a brief review of the literature. To establish the inequalities, our approach is based on a completion of the square method which accounts for duality issues. Linearization (spectral estimates) and estimates based on a nonlinear flow are used for optimality issues. Although some of these methods have been widely used in the literature, for instance in the context of Hardy inequalities (see \cite{MR2379440} and references therein), it seems that they have not been fully exploited yet in the case of the functional inequalities considered in this paper. The main tool in~\cite{Dolbeault2011} is a flow of fast diffusion type, which has been considered earlier in~\cite{delPino-Saez01}. In dimension $d=2$, we may refer to various papers (see for instance \cite{MR1679782,MR1371208,MR2610890}) in connection with Ricci's flow for properties of the solutions of the corresponding evolution equation.

Many papers have been devoted to the asymptotic behaviour near extinction of the solutions of nonlinear flows, in bounded domains (see for instance \cite{MR588035,MR1877973,MR1285092,BGV10}) or in the whole space (see \cite{king1993self,MR1348964,MR1475779} and references therein). In particular, the Cauchy-Schwarz inequality has been repeatedly used, for instance in \cite{MR588035,MR1285092}, and turns out to be a key tool in the main result of \cite{Dolbeault2011}, as well as the solution \emph{with separation of variables,} which is related to the Aubin-Talenti optimal function for~\eqref{Ineq:Sobolev}.

Getting improved versions of Sobolev's inequality is a question which has attracted lots of attention.
See \cite{MR790771} in the bounded domain case and~\cite{BN83} for an earlier related paper. However, in~\cite{MR790771}, H.~Brezis and E.~Lieb also raised the question of measuring the distance to the manifold of optimal functions in the case of the Euclidean space. A few years later, G.~Bianchi and H.~Egnell gave an answer in \cite{MR1124290} using the concentration-compactness method, with no explicit value of the constant. Since then, considerable efforts have been devoted to obtain quantitative improvements of Sobolev's inequality. On the whole Euclidean space, nice estimates based on rearrangements have been obtained in \cite{MR2538501} and we refer to \cite{MR2508840} for an interesting review of various related results. The method there is in some sense constructive, but it hard to figure what is the practical value of the constant. As in~\cite{Dolbeault2011} our approach involves much weaker notions of distances to optimal functions, but on the other hand offers clear-cut estimates. Moreover, it provides an interesting way of obtaining global estimates based on a linearization around Aubin-Talenti optimal functions.

\section{A completion of the square and consequences}\label{Sec:Square}

Before proving the main results of this paper, let us explain in which sense Sobolev's inequality and the Hardy-Littlewood-Sobolev inequality, or Onofri's inequality and the logarithmic Hardy-Littlewood-Sobolev inequality, for instance, are \emph{dual inequalities}.

To a convex functional $F$, we may associate the functional $F^*$ defined by Legendre's duality as
\[
F^*[v]:=\sup\(\ird{u\,v}-F[u]\)\,.
\]
For instance, to $F_1[u]=\frac{1}2\,\nrm up^2$
defined on $\L^p(\R^d)$, we henceforth associate $F_1^*[v]=\frac{1}2\,\nrm vq^2$ on $\L^q(\R^d)$ where $p$ and $q$ are H\"older conjugate exponents: $1/p +1/q=1$. The supremum can be taken for instance on all functions in $\L^p(\R^d)$, or, by density, on the smaller space of the functions $u\in\L^p(\R^d)$ such that $\nabla u\in\L^2(\R^d)$. Similarly, to $F_2[u]=\frac{1}2\,\mathsf{S}_d\,\nrm{\nabla u}2^2$, we associate $F_2^*[v]=\frac{1}2\,\mathsf{S}_d^{-1}\ird{v\,(-\Delta)^{-1}\,v}$ where $(-\Delta)^{-1}\,v=G_d*v$ with $G_d(x)=\frac{1}{d-2}\,|\mathbb S^{d-1}|^{-1}\,|x|^{2-d}$, when $d\ge3$, and $G_2(x)=-\,\frac1{2\pi}\,\log|x|$. As a straightforward consequence of Legendre's duality, if we have a functional inequality of the form $F_1[u]\le F_2[u]$, then we have the \emph{dual inequality} $F_1^*[v]\ge F_2^*[v]$. In this sense,~\eqref{Ineq:Sobolev} and~\eqref{Ineq:HLS} are dual of each other, as it has been noticed in \cite{MR717827}. Also notice that Inequality~\eqref{Ineq:HLS} is a consequence of Inequality~\eqref{Ineq:Sobolev}.

In this paper, we go one step further and establish that
\be{Ineq:Abstract}
F_1^*[u]-F_2^*[u]\le\mathsf{C}\(F_2[u]-F_1[u]\)
\ee
for some positive constant $\mathsf{C}$, at least under some normalization condition (or up to a multiplicative term which is required for simple homogeneity reasons). Such an inequality has been established in \cite[Theorem~1.2]{Dolbeault2011} when $d\ge5$. Here we extend it to any $d\ge3$ and get and improved value for the constant $\mathsf{C}$.

It turns out that the proof can be reduced to the completion of a square. Let us explain how the method applies in case of Theorem~\ref{Thm:SquareSobolev}, and how Theorem~\ref{Thm:Onofri} can be seen as a limit of Theorem~\ref{Thm:SquareSobolev} in case of radial functions.

\begin{proof}[Proof of Theorem~\ref{Thm:SquareSobolev}, part 1: the completion of a square]$\phantom{x}$\\ Integrations by parts show that
\[
\ird{|\nabla(-\Delta)^{-1}\,v|^2}=\ird{v\,(-\Delta)^{-1}\,v}
\]
and, if $v=u^q$ with $q=\frac{d+2}{d-2}$,
\[
\ird{\nabla u\cdot\nabla(-\Delta)^{-1}\,v}=\ird{u\,v}=\ird{u^{2^*}}\,.
\]
Hence the expansion of the square
\[
0\le\ird{\left|\mathsf{S}_d\,\nrm u{2^*}^\frac4{d-2}\,\nabla u-\nabla(-\Delta)^{-1}\,v\right|^2}
\]
shows that
\begin{multline*}
0\le\mathsf{S}_d\,\nrm u{2^*}^\frac{8}{d-2}\left[\mathsf{S}_d\,\nrm{\nabla u}2^2-\nrm u{2^*}^2\right]\\
-\Big[\mathsf{S}_d\,\nrm{u^q}{\frac{2\,d}{d+2}}^2-\ird{u^q\,(-\Delta)^{-1}\,u^q}\Big].
\end{multline*}
Equality is achieved if and only if
\[
\mathsf{S}_d\,\nrm u{2^*}^\frac4{d-2}\,u=(-\Delta)^{-1}\,v=(-\Delta)^{-1}\,u^q\,,
\]that is, if and only if $u$ solves
\[
-\,\Delta u=\frac{1}{\mathsf{S}_d}\,\nrm u{2^*}^{-\frac4{d-2}}\,u^q\,,
\]
which means that $u$ is an Aubin-Talenti function, optimal for \eqref{Ineq:Sobolev}. This completes the proof of Theorem~\ref{Thm:SquareSobolev}, up to the optimality of the proportionality constant, for which we know that
\be{Ineq:Cupper}
\mathsf{C}_d=\mathcal C\,\mathsf{S}_d\quad\mbox{with}\quad\mathcal C\le1\,.
\ee
Incidentally, this also proves that $v$ is optimal for \eqref{Ineq:HLS}.\end{proof}

As a first step towards the proof of Theorem~\ref{Thm:Onofri}, let us start with a result for radial functions. If $d$ is a positive integer, we can define
\[
\mathsf{s}_d:=\mathsf{S}_d\,|\S^{d-1}|^\frac2d
\]
and get
\be{Eqn:sd}
\mathsf{s}_d=\frac4{d\,(d-2)}\(\frac{\Gamma\(\frac{d+1}2\)}{\sqrt\pi\,\Gamma\(\frac{d}2\)}\)^\frac2d\,.
\ee
Using this last expression allows us to consider $d$ as a real parameter.
\begin{lem}\label{lem:radial} Assume that $d\in\R$ and $d>2$. Then
\begin{multline*}\label{S-HLS}\hspace*{-10pt}
0\le\mathsf{s}_d\(\int_0^\infty u^\frac{2\,d}{d-2}\;r^{d-1}\,dr\)^{1+\frac{2}d}-\int_0^\infty u^\frac{d+2}{d-2}\((-\Delta)^{-1}u^\frac{d+2}{d-2}\)\,r^{d-1}\,dr\\
\le\mathsf{c}_d\(\int_0^\infty u^\frac{2\,d}{d-2}\;r^{d-1}\,dr\)^\frac{4}d\left[\,\mathsf{s}_d\!\int_0^\infty|u'|^2\;r^{d-1}\,dr-\(\int_0^\infty\!u^\frac{2\,d}{d-2}\;r^{d-1}\,dr\)^\frac{d-2}d\right]
\end{multline*}
holds for any radial function $u\in\D$ with optimal constant $\mathsf{c}_d\le\mathsf{s}_d$.\end{lem}
Here we use the notation $(-\Delta)^{-1}\,v=w$ to express the fact that $w$ is the solution to $w''+\frac{d-1}r\,w'+v=0$, that is,
\be{Def:InvLap}
(-\Delta)^{-1}\,v\,(r)=\int_r^\infty s^{1-d}\int_0^sv(t)\;t^{d-1}\,dt\;ds\quad\forall\,r>0\,.
\ee

\begin{proof} In the case of a radially symmetric function $u$, and with the standard abuse of notations that amounts to identify $u(x)$ with $u(r)$, $r=|x|$, Inequality~\eqref{Ineq:Sobolev} can be written as
\be{Ineq:SobolevRadial}
\mathsf{s}_d\int_0^\infty |u'|^2\,r^{d-1}\,dr\ge\(\int_0^\infty |u|^\frac{2\,d}{d-2}\,r^{d-1}\,dr\)^{1-\frac2d}\,.
\ee
However, if $u$ is considered as a function of one real variable $r$, then the inequality also holds for any \emph{real parameter} $d\in(2,\infty)$ and is equivalent to the one-dimensional Gagliardo-Nirenberg inequality
\[
\mathsf{s}_d\(\irt{|w'|^2}+\tfrac14\,(d-2)^2\irt{|w|^2}\)\ge\(\irt{|w|^\frac{2\,d}{d-2}}\)^{1-\frac2d}
\]
as can be shown using the Emden-Fowler transformation
\be{Eqn:Emden-Fowler}
u(r)=(2\,r)^{-\frac{d-2}2}\,w(t)\,,\quad t=-\,\log r\,.
\ee
The corresponding optimal function is, up to a multiplication by a constant, given by
\[
w_\star(t)=\(\cosh t\)^{-\frac{d-2}2}\quad\forall\,t\in\R\,,
\]
which solves the Euler-Lagrange equation
\[
-\,(p-2)^2\,w''+4\,w-\,2\,p\,|w|^{p-2}\,w=0\,.
\]
for any real number $d>2$ and the optimal function for~\eqref{Ineq:SobolevRadial} is
\[
u_\star(r)=(2\,r)^{-\frac{d-2}2}\,w_\star(-\log r)=\(1+r^2\)^{-\frac{d-2}2}
\]
up to translations, multiplication by a constant and scalings. This establishes~\eqref{Eqn:sd}. See \ref{Sec:UsefulFormulae} for details on the computation of $\mathsf{s}_d$. The reader is in particular invited to check that the expression of $\mathsf{s}_d$ is consistent with the one of $\mathsf{S}_d$ given in the introduction.

Next we apply Legendre's transform to \eqref{Ineq:SobolevRadial} and get a Hardy-Littlewood-Sobolev inequality that reads
\be{Ineq:HLSradial}
\int_0^\infty v\;(-\Delta)^{-1}\,v\;r^{d-1}\,dr\le\mathsf{s}_d\(\int_0^\infty v^\frac{2\,d}{d+2}\;r^{d-1}\,dr\)^{1+\frac{d}2}
\ee
for any $d>2$. Inequality~\eqref{Ineq:HLSradial} holds on the functional space which is obtained by completion of the space of smooth compactly supported radial functions with respect to the norm defined by the r.h.s.~in~\eqref{Ineq:HLSradial}. Inequality~\eqref{Ineq:HLSradial} is the first inequality of Lemma~\ref{lem:radial}.

Finally, we apply the completion of the square method. By expanding
\[
0\le\int_0^\infty\big|\,a\,u'-\big((-\Delta)^{-1}v\big)'\,\big|^2\;r^{d-1}\,dr
\]
with $a=\mathsf{s}_d\,\(\int_0^\infty u^\frac{2\,d}{d-2}\;r^{d-1}\,dr\)^\frac{2}d$ and $v=u^\frac{d-2}{d+2}$, we establish the second inequality of Lemma~\ref{lem:radial} (with optimal constant $\mathsf{c}_d\le\mathsf{s}_d$).
\end{proof}

Now let us turn our attention to the case $d=2$ and to Theorem~\ref{Thm:Onofri}. Using the fact that $d$ in Lemma~\ref{lem:radial} is a real parameter, we can simply consider the limit of the inequalities as $d\to2_+$.
\begin{cor}\label{cor:radial} For any function $f\in\L^1(\R^+;\,r\,dr)$ such that $f'\in\L^2(\R^+;\,r\,dr)$ and $M=\int_0^\infty e^f\,(1+r^2)^{-2}\,2\,r\,dr$, we have the inequality
\begin{multline}\label{S-logHLS}\hspace*{-10pt}
0\le\int_0^\infty e^f\,\log\(\frac{e^f}{M\,(1+r^2)^2}\)\;\frac{2\,r\,dr}{(1+r^2)^2}\\
\quad-\,\frac2M\int_0^\infty\frac{e^f}{(1+r^2)^2}\;(-\Delta)^{-1}\(\frac{e^f}{(1+r^2)^2}\)\;2\,r\,dr+\,M\\
\le M\,\left[\frac18\int_0^\infty |f'|^2\,r\,dr+\int_0^\infty f\;\frac{2\,r\,dr}{(1+r^2)^2}-\log\(\int_0^\infty e^f\;\frac{2\,r\,dr}{(1+r^2)^2}\)\right]\,.
\end{multline}
\end{cor}
Here again $(-\Delta)^{-1}$ is defined by~\eqref{Def:InvLap}, but it coincides with the inverse of $-\Delta$ acting on radial functions.

\begin{proof}
We may pass to the limit in~\eqref{Ineq:SobolevRadial} written in terms
of
\[
u(r)=u_\star(r)\(1+\tfrac{d-2}{2\,d}\,f\)
\]
to get the radial version of Onofri's inequality for $f$. By expanding the expression of $|u'|^2$ we get
\[
u'^2=u_\star'^2+\frac{d-2}d\,u_\star'\(u_\star\,f\)'+\(\frac{d-2}{2\,d}\)^2\(u_\star'\,f+u_\star\,f'\)^2\,.
\]

Using the fact that $\lim_{d\to2_+}(d-2)\,\mathsf{s}_d=1$,
\[
\mathsf{s}_d=\frac1{d-2}+\frac{1}2-\frac{1}2\,\log 2+o(1)\quad\mbox{as}\quad d\to2_+\,,
\]
and
\[
\lim_{d\to2_+}\frac1{d-2}\int_0^\infty|u_\star'|^2\;r^{d-1}\,dr=1\,,
\]
\[
\frac1{d-2}\int_0^\infty|u_\star'|^2\;r^{d-1}\,dr-1\sim-\frac12\,(d-2)\,,
\]
\[
\lim_{d\to2_+}\frac1{d-2}\int_0^\infty u_\star'\(u_\star\,f\)'\,r^{d-1}\,dr=\int_0^\infty f\;\frac{2\,r\,dr}{(1+r^2)^2}\,,
\]
\[
\lim_{d\to2_+}\frac1{4\, d^2}\int_0^\infty |f'|^2\,u_\star^2\,r^{d-1}\,dr=\frac1{16}\int_0^\infty |f'|^2\,r\,dr\,,
\]
and finally
\[
\lim_{d\to2_+}\int_0^\infty |u_\star\,(1+\tfrac{d-2}{2\,d}\,f)|^\frac{2\,d}{d-2}\,r^{d-1}\,dr=\int_0^\infty e^f\;\frac{r\,dr}{(1+r^2)^2}\,,
\]
so that, as $d\to2_+$,
\[
\(\int_0^\infty |u_\star\,(1+\tfrac{d-2}{2\,d}\,f)|^\frac{2\,d}{d-2}\,r^{d-1}\,dr\)^\frac{d-2}d-1
\sim\frac{d-2}2\,\log\(\int_0^\infty e^f\;\frac{r\,dr}{(1+r^2)^2}\).
\]
By keeping only the highest order terms, which are of the order of $(d-2)$, and passing to the limit as $d\to2_+$ in~\eqref{Ineq:SobolevRadial}, we obtain that
\[
\frac18\int_0^\infty |f'|^2\,r\,dr+\int_0^\infty f\;\frac{2\,r\,dr}{(1+r^2)^2}\ge\log\(\int_0^\infty e^f\;\frac{2\,r\,dr}{(1+r^2)^2}\)\,,
\]
which is Onofri's inequality written for radial functions.

Similarly, we can pass to the limit as $d\to2_+$ in~\eqref{Ineq:HLSradial}. Let $v$ be a compactly supported smooth radial function, considered as a function of $r\in[0,\infty)$ and let us compute the limit as $d\to2_+$ of
\[
\mathsf{h}(d):=\(\int_0^\infty v^\frac{2\,d}{d+2}\;r^{d-1}\,dr\)^{1+\frac{2}d}-\frac1{\mathsf{s}_d}\int_0^\infty v\,k_d[v]\,r^{d-1}\,dr
\]
where $k_d[v]:=(-\Delta)^{-1}\,v$ is given by \eqref{Def:InvLap} for any $d\ge2$. If $d>2$, since
\begin{multline*}
(2-d)\int_0^\infty v(r)\,k_d[v](r)\,r^{d-1}\,dr\\
=(2-d)\int_0^\infty v(r)\,r^{d-1}\int_r^\infty s^{1-d}\int_0^sv(t)\;t^{d-1}\,dt\;ds\;dr\\
=(2-d)\int_0^\infty r^{1-d}\(\int_0^r v(t)\;t^{d-1}\,dt\)^2dr\\
=-\,2\int_0^\infty r\,v(r)\int_0^r v(t)\;t^{d-1}\,dt\;dr
\end{multline*}
we see that $\lim_{d\to2_+}\mathsf{h}(d)=0$ since
\[
2\int_0^\infty r\,v(r)\int_0^r v(t)\;t\,dt\;dr=\(\int_0^\infty r\,v(r)\;dr\)^2\,.
\]
Let us compute the $O(d-2)$ term. With the above expression, it is now easy to check that
\begin{eqnarray*}
&&\lim_{d\to2_+}\frac{\mathsf{h}(d)}{d-2}\\
&&=\frac{1}2\int_0^\infty v\,r\,dr\int_0^\infty v\,\log\(\frac{v}{\int_0^\infty v\,r\,dr}\)r\,dr-\,\frac{\log2-1}2\(\int_0^\infty r\,v(r)\;dr\)^2\\
&&\hspace*{12pt}+\,2\int_0^\infty v\,r\,dr\int_0^\infty v(r)\;r\log r\;dr-2\int_0^\infty r\,v(r)\int_0^r v(t)\;t\log t\;dt\;dr\\
&&=\frac{1}2\int_0^\infty v\,r\,dr\int_0^\infty v\,\log\(\frac{v}{\int_0^\infty v\,r\,dr}\)r\,dr-\,\frac{\log2-1}2\(\int_0^\infty r\,v(r)\;dr\)^2\\
&&\hspace*{12pt}+\,2\int_0^\infty v\,r\,dr\int_r^\infty v(t)\;t\log t\;dt
\end{eqnarray*}
since $\frac1{(d-2)\,\mathsf{s}_d}\sim1+\frac{d-2}2\,(\log2-1)$. A computation corresponding to $d=2$ similar to the one done above for $d>2$ shows that, when $d=2$,
\begin{multline*}
\int_0^\infty v\,k_2[v]\,r\,dr
=\int_0^\infty v(r)\,r\int_r^\infty\frac1s\int_0^sv(t)\;t\,dt\;ds\;dr\\
=\int_0^\infty\frac1r\(\int_0^r v(t)\;t\,dt\)^2dr\\
=\,-\,2\int_0^\infty r\,\log r\,v(r)\int_0^r v(t)\,t\,dt\;dr\,,
\end{multline*}
thus proving that
\begin{multline*}
\lim_{d\to2_+}\frac{\mathsf{h}(d)}{d-2}=\frac{1}2\int_0^\infty v\,r\,dr\int_0^\infty v\,\log\(\frac{v}{\int_0^\infty v\,r\,dr}\)r\,dr-\int_0^\infty v\,k_2[v]\,r^{d-1}\,dr\\
-\,\frac12\,(\log2-1)\(\int_0^\infty r\,v(r)\;dr\)^2\,.
\end{multline*}

Now let us consider as above the limit
\[
u^\frac{d+2}{d-2}=(1+r^2)^{-\frac{d+2}2}\,(1+\tfrac{d-2}{2\,d}\,f)^\frac{d+2}{d-2}\to(1+r^2)^{-2}\,e^f=:g
\]
as $d\to2$. This concludes the proof of Corollary~\ref{cor:radial} by passing to the limit in the inequalities of Lemma~\ref{lem:radial} and taking $v=g$.\end{proof}

\begin{proof}[Proof of Theorem~\ref{Thm:Onofri}: a passage to the limit in the radial case]
If we consider $g$ as a function on $\R^2\ni x$ with $r=|x|$, this means that
\begin{multline*}
\lim_{d\to2_+}\frac{\mathsf{h}(d)}{d-2}=\frac{1}2\irdeux g\irdeux{g\,\log\(\frac{g}{\irdeux g}\)}-2\,\pi\irdeux{g\;(-\Delta)^{-1}\,g}\\
+\frac12\,(1+\log\pi)\(\irdeux g\)^2
\end{multline*}
which precisely corresponds to the terms involved in \eqref{Ineq:logHLS}, up to a factor $\frac{1}2\,M=\frac12\irdeux g$. The proof in the non-radial case will be provided at the end of Section~\ref{Sec:CKN}.\end{proof}

\section{Linearization}\label{Sec:Linearization}

In the previous section, we have proved that the optimal constant $\mathsf{C}_d$ in~\eqref{S-HLS} is such that $\mathsf{C}_d \le \mathsf{S}_d$. Let us prove that $\mathsf{C}_d\ge\frac{d}{d+4}\,\mathsf{S}_d$ using a special sequence of test functions. Let $\mathcal F$ and $\mathcal G$ be the positive integral quantities associated with, respectively, the Sobolev and Hardy-Littlewood-Sobolev inequalities:
\[
    \mathcal{F}[u] := \mathsf{S}_d\,\nrm{\nabla u}2^2-\nrm u{2^*}^2\,,
\]
\[
    \mathcal{G}[v] := \mathsf{S}_d\,\nrm{v}{\frac{2\,d}{d+2}}^2-\ird{v\,(-\Delta)^{-1}\,v}\,.
\]
Since that, for the Aubin-Talenti extremal function $u_\star$, we have~$\mathcal F[u_\star] = \mathcal G[u_\star^q] = 0$, so that $u_\star$ gives a case of equality for \eqref{S-HLS}, a natural question to ask is whether the infimum of $\mathcal F[u]/\mathcal G[u^q]$, under an appropriate normalization of~$\nrm u{2^*}$, is achieved as a perturbation of the $u_\star$.

Recall that $u_\star$ is the Aubin-Talenti extremal function
\[
u_\star(x):=(1+|x|^2)^{-\frac{d-2}2}\quad\forall\,x\in\R^d\,.
\]
With a slight abuse of notations, we use the same notation as in Section~\ref{Sec:Square}. We may notice that $u_\star$ solves
\[
-\Delta u_\star=d\,(d-2)\,u_\star^\frac{d+2}{d-2}
\]
which allows to compute the optimal Sobolev constant as
\be{Eqn:OptSob}
\mathsf{S}_d=\frac1{d\,(d-2)}\,\(\ird{u_\star^{2^*}}\)^{-\frac2d}
\ee
using \eqref{Eqn:Emden-Fowler}. See \ref{Sec:UsefulFormulae} for details. This shows that
\[
\frac1{\mathsf{S}_d}\,\mathcal F[u]=\nrm{\nabla u}2^2-\,d\,(d-2)\(\ird{u^{2^*}}\)^{1-\frac{2}d}\(\ird{u_\star^{2^*}}\)^\frac{2}d.
\]

The goal of this section is to perform a linearization. By expanding $\mathcal F[u_\varepsilon]$ with $u_\varepsilon=u_\star+\varepsilon\,f$, for some $f$ such that $\ird{\frac{f\,u_\star}{(1+|x|^2)^2}}=0$ at order two in terms of $\varepsilon$, we get that
\[
\frac1{\mathsf{S}_d}\,\mathcal F[u_\varepsilon]=\varepsilon^2\,\mathsf{F}[f]+o(\varepsilon^2)
\]
where
\[
\mathsf{F}[f]:=\ird{|\nabla f|^2}-\,d\,(d+2)\ird{\frac{|f|^2}{(1+|x|^2)^2}}\,.
\]
According to Lemma~\ref{Lem:Poincare} (see \ref{Sec:Stereographic}), we know that
\[
\mathsf{F}[f]\ge4\,(d+2)\ird{\frac{|f|^2}{(1+|x|^2)^2}}
\]
for any $f\in\mathcal D^{1,2}(\R^d)$ such that
\be{Eqn:Orthogonality}
\ird{\frac{f\,f_i}{(1+|x|^2)^2}}=0\quad\forall\,i=0\,,\;1\,,\;2\,,\;\ldots d+1\,,
\ee
where
\[
f_0:=u_\star\,,\quad f_i(x)=\frac{x_i}{1+|x|^2}\,u_\star(x)\quad\mbox{and}\quad f_{d+1}(x):=\frac{1-|x|^2}{1+|x|^2}\,u_\star(x)\,.
\]
Notice for later use that
\[
-\,\Delta f_0=d\,(d-2)\,\frac{f_0}{(1+|x|^2)^2}
\]
and
\[
-\,\Delta f_i=d\,(d+2)\,\frac{f_i}{(1+|x|^2)^2}\quad\forall\,i=1\,,\;2\,,\;\ldots d+1\,.
\]
Also notice that
\[
\ird{\frac{f_i\,f_j}{(1+|x|^2)^2}}=0
\]
for any $i$, $j=0$, $1$, \ldots $d+1$, $j\neq i$.

Similarly, we can consider the functional $\mathcal G$ as given above, associated with the Hardy-Little\-wood-Sobolev inequality,
and whose minimum $\mathcal G[v_\star]=0$ is achieved by $v_\star:=u_\star^q$, $q=\frac{d+2}{d-2}$. Consistently with the above computations, let  $v_\varepsilon:=\(u_\star+\varepsilon\,f\)^q=v_\star\,\big(1+\varepsilon\,\tfrac f{u_\star}\big)^q$ where $f$ is such that $\ird{\frac{f\,f_0}{(1+|x|^2)^2}}=0$. By expanding $\mathcal G[v_\varepsilon]$ at order two in terms of~$\varepsilon$, we get that
\[
\mathcal G[v_\varepsilon]=\varepsilon^2\,\(\frac{d+2}{d-2}\)^2\,\mathsf{G}[f]+o(\varepsilon^2)
\]
where
\begin{multline*}
\mathsf{G}[f]:=\frac1{d\,(d+2)}\,\ird{\frac{|f|^2}{(1+|x|^2)^2}}\\
-\ird{\frac{f}{(1+|x|^2)^2}\;(-\Delta)^{-1}\!\(\frac{f}{(1+|x|^2)^2}\)}\,.
\end{multline*}
\begin{lem}\label{lem:linearization3} $\mathrm{Ker}(\mathsf{F})=\mathrm{Ker}(\mathsf{G})$.\end{lem}
It is straightforward to check that the kernel is generated by $f_i$ with $i=1$, $2$, \ldots $d$, $d+1$. Details are left to the reader.
Next, by Legendre duality we find that
\[
\frac12\ird{\frac{|g|^2}{(1+|x|^2)^2}}=\sup_f\(\ird{\frac{f\,g}{(1+|x|^2)^2}}-\,\frac12\ird{\frac{|f|^2}{(1+|x|^2)^2}}\)\,,
\]
\begin{multline*}
\frac12\ird{\frac{g}{(1+|x|^2)^2}\,(-\Delta)^{-1}\(\frac{g}{(1+|x|^2)^2}\)}\\
=\sup_f\(\ird{\frac{f\,g}{(1+|x|^2)^2}}-\,\frac12\ird{|\nabla f|^2}\)\,.
\end{multline*}
Here the supremum is taken for all $f$ satisfying the orthogonality conditions~\eqref{Eqn:Orthogonality}. It is then straightforward to see that duality holds if $g$ is restricted to functions satisfying~\eqref{Eqn:Orthogonality} as well. Consider indeed an optimal function $f$ subject to~\eqref{Eqn:Orthogonality}. There are Lagrange multipliers $\mu_i\in\R$ such that
\[
g-f-\sum_{i=0}^{d+1}\mu_i\,f_i=0
\]
and after multiplying by $f\,(1+|x|^2)^{-2}$, an integration shows that
\[
\ird{\frac{f\,g}{(1+|x|^2)^2}}=\ird{\frac{|f|^2}{(1+|x|^2)^2}}
\]
using the fact that $f$ satisfies~\eqref{Eqn:Orthogonality}. On the other hand, if $g$ satisfies~\eqref{Eqn:Orthogonality}, after multiplying by $g\,(1+|x|^2)^{-2}$, an integration gives
\[
\ird{\frac{|g|^2}{(1+|x|^2)^2}}=\ird{\frac{f\,g}{(1+|x|^2)^2}}\,,
\]
which establishes the first identity of duality. As for the second identity, the optimal function satisfies the Euler-Lagrange equation
\[
\frac{g}{(1+|x|^2)^2}+\,\Delta\,f=\sum_{i=0}^{d+1}\mu_i\,\frac{f_i}{(1+|x|^2)^2}
\]
for some Lagrange multipliers that we again denote by $\mu_i$. By multiplying by $f$ and $(-\Delta)^{-1}\big(g\,(1+|x|^2)^{-2}\big)$, we find that
\begin{eqnarray*}
&&
\ird{\frac{f\,g}{(1+|x|^2)^2}}=\ird{|\nabla f|^2}\\
&&\ird{\frac{g}{(1+|x|^2)^2}\,(-\Delta)^{-1}\(\frac{g}{(1+|x|^2)^2}\)}=\ird{\frac{f\,g}{(1+|x|^2)^2}}
\end{eqnarray*}
where we have used the fact that 
\begin{multline*}
\ird{\frac{f_i}{(1+|x|^2)^2}\,(-\Delta)^{-1}\(\frac{g}{(1+|x|^2)^2}\)}\\
=\ird{\frac{g}{(1+|x|^2)^2}\,(-\Delta)^{-1}\(\frac{f_i}{(1+|x|^2)^2}\)}=0
\end{multline*}
because $(-\Delta)^{-1}\big(f_i\,(1+|x|^2)^{-2}\big)$ is proportional to $f_i$. As a straightforward consequence, the dual form of Lemma~\ref{Lem:Poincare} then reads as follows.
\begin{cor}\label{cor:linearization4} For any $g$ satisfying the orthogonality conditions~\eqref{Eqn:Orthogonality}, we have
\[
\ird{\frac{g}{(1+|x|^2)^2}\,(-\Delta)^{-1}\!\(\frac{g}{(1+|x|^2)^2}\)\!}\le\frac1{(d+2)\,(d+4)}\!\ird{\frac{g^2}{(1+|x|^2)^2}}\,.
\]
Moreover, if $f$ obeys to~\eqref{Eqn:Orthogonality}, then we have
\[
\frac4{d\,(d+2)\,(d+4)}\ird{\frac{f^2}{(1+|x|^2)^2}}\le\mathsf{G}[f]\le\frac1{d\,(d+2)^2\,(d+4)}\,\mathsf{F}[f]
\]
and equalities are achieved in $\L^2(\R^d,(1+|x|^2)^{-2}\,dx)$.\end{cor}
\begin{proof} The first inequality follows from the above considerations on duality and the second one from the definition of $\mathsf{G}$, using
\[
\frac4{d\,(d+2)\,(d+4)}=\frac1{d\,(d+2)}-\frac1{(d+2)\,(d+4)}\,.
\]
To establish the last inequality, we can decompose $f$ on $(f_k)_k$, the stereographic projection of the spherical harmonics associated to eigenvalues $\lambda_k=k\,(k+d-1)$ with $k\ge2$, so as to meet condition~\eqref{Eqn:Orthogonality}. See~\ref{Sec:Stereographic} for more details. The corresponding eigenvalues for the Laplacian operator on the Euclidean space
are $\mu_k=4\,\lambda_k+\,d\,(d-2)$, so that $-\Delta f_k=\mu_k\,f_k\,(1+|x|^2)^{-2}$, with $\|f_k\|_{L^2\left(\R^d,\,(1+|x|^2)^{-2}\,dx\right)}=1$. By writing $f=\sum_{k\ge2}a_k\,f_k$ we have
\[
\mathsf{F}[f]=\sum_{k\ge2}c_k\,,\quad\mbox{with}\quad c_k:=a_k^2\,(\mu_k-\mu_1)\,,
\]
\[
\mathsf{G}[f]=\sum_{k\ge2}d_k\,,\quad\mbox{with}\quad d_k:=a_k^2\(\frac1{\mu_1}-\frac1{\mu_k}\)\,,
\]
with $c_k = \mu_1\, \mu_k\, d_k \le \mu_1\,\mu_2\, d_k$ since $(\mu_k)_k$ is increasing in $k$. This yields
\[
\frac{\mathsf{F}[f]}{\mathsf{G}[f]}\le\mu_1\,\mu_2=d\,(d+2)^2\,(d+4)\,,
\]
with equality for $f=f_2$.
\end{proof}

As a consequence of Corollary~\ref{cor:linearization4} and \eqref{Eqn:OptSob}, we have found that
\be{Final:Linearization}
\frac1{\mathcal C}:=\frac{\mathsf{S}_d}{\mathsf{C}_d}=\inf_{\mathcal G[u^q]\neq0}\frac{\nrm u{2^*}^\frac{8}{d-2}\,\mathsf{S_d}\,\mathcal F[u]}{\mathcal G[u^q]}\le\frac1{d^2\,(d+2)^2}\,\inf_f\frac{\mathsf{F}[f]}{\mathsf{G}[f]}=\frac{d+4}d\,,
\ee
where the last infimum is taken on the set of all non-trivial functions in $\L^2(\R^d,(1+|x|^2)^{-2}\,dx)$ satisfying~\eqref{Eqn:Orthogonality}. This establishes the lower bound in~\eqref{S-HLS}. 
\begin{remark} One may hope to get a better estimate by considering the case $f\in\mathrm{Ker}(\mathsf{F}) = \mathrm{Ker}(\mathsf{G})$ and expanding $\mathcal F$ and $\mathcal G$ to the fourth order in $\varepsilon$ but, interestingly, this yields exactly the same lower bound on $\mathsf{C}_d$ as the linearization shown above.\end{remark}

\section{Improved inequalities and nonlinear flows}\label{Sec:Improvements}

In Section~\ref{Sec:Linearization}, the basic strategy was based on the completion of a square. The initial approach for the improvement of Sobolev inequalities in \cite{Dolbeault2011} was based on a fast diffusion flow. Let us give some details and explain how even better results can be obtained using a combination of the two approaches.

Let us start with a summary of the method of \cite{Dolbeault2011}. It will be convenient to define the functionals
\[
\mathsf{J}_d[v]:=\ird{v^\frac{2\,d}{d+2}}\quad\mbox{and}\quad\mathsf{H}_d[v]:=\ird{v\,(-\Delta)^{-1}v}-\mathsf{S}_d\,\nrm v{\frac{2\,d}{d+2}}^2\,.
\]
Consider a positive solution $v$ of the \emph{fast diffusion} equation
\be{Eqn:FD}
\frac{\partial v}{\partial t}=\Delta v^m\quad t>0\,,\quad x\in\R^d\,,\quad m=\frac{d-2}{d+2}
\ee
and define the functions
\[
\mathsf{J}(t):=\mathsf{J}_d[v(t,\cdot)]\quad\mbox{and}\quad\mathsf{H}(t):=\mathsf{H}_d[v(t,\cdot)]\,.
\]
We shall denote by $\mathsf{J}_0$ and $\mathsf{H}_0$ the corresponding initial values. Elementary computations show that
\be{Eqn:DynSobolev}
\mathsf{J}'=-\,(m+1)\,\nrm{\nabla v^m}2^2\le -\,\frac{m+1}{\mathsf{S}_d}\,\mathsf{J}^{1-\frac{2}d}=-\,\frac{2\,d}{d+2}\frac1{\mathsf{S}_d}\,\mathsf{J}^{1-\frac{2}d}\,,
\ee
where the inequality is a consequence of Sobolev's inequality. Hence $v$ has a finite extinction time $T>0$ and since
\[
\mathsf{J}(t)^\frac{2}d\le\mathsf{J}_0^\frac{2}d-\frac4{d+2}\,\frac{t}{\mathsf{S}_d}\,,
\]
we find that
\[
T\le\frac{d+2}4\,\mathsf{S}_d\,\mathsf{J}_0^\frac{2}d\,.
\]
We notice that $\mathsf{H}$ is nonpositive because of the Hardy-Littlewood-Sobolev inequality and by applying the flow of \eqref{Eqn:FD}, we get that
\[
\frac12\,\mathsf{J}^{-\frac{2}d}\,\mathsf{H}'=\mathsf{S}_d\,\nrm{\nabla u}2^2-\nrm u{2^*}^2\quad\mbox{with}\quad u=v^\frac{d-2}{d+2}\,.
\]
The right hand side is nonnegative because of Sobolev's inequality. One more derivation with respect to $t$ gives that
\be{Eqn:K}
\mathsf{H}''=\frac{\mathsf{J}'}{\mathsf{J}}\,\mathsf{H}'-4\,m\,\mathsf{S}_d\,\mathsf{J}^\frac{2}d\,\mathsf{K}
\ee
where $\mathsf{K}:=\ird{v^{m-1}\,|(-\Delta) v^m-\Lambda\,v|^2}$ and $\Lambda:=-\frac{d+2}{2\,d}\,\frac{\mathsf{J}'}{\mathsf{J}}$. This identity makes sense in dimension $d\ge5$, because, close to the extinction time, $v$ behaves like the Aubin-Talenti functions. The reader is invited to check that all terms are finite when expanding the square in $\mathsf{K}$ and can refer to \cite{Dolbeault2011} for more details. It turns out that the following estimate is also true if $d=3$ or $d=4$.
\begin{lem}\label{Lem:Concavity} Assume that $d\ge3$. With above notations, we have
\[
\frac{\mathsf{H}''}{\mathsf{H}'}\le\frac{\mathsf{J}'}{\mathsf{J}}\,.
\]
\end{lem}
The main idea is that even if each of the above integrals is infinite, there are cancellations in low dimensions. To clarify this computation, it is much easier to get rid of the time-dependence corresponding to the \emph{solution with separation of variables} and use the inverse stereographic projection to recast the problem on the sphere. The sketch of the proof of this lemma will be given in~\ref{Sec:FlowSphere}.

A straightforward consequence is the fact that
\[
\frac{\mathsf{H}''}{\mathsf{H}'}\le-\,\kappa\quad\mbox{with}\quad\kappa:=\frac{2\,d}{d+2}\,\frac{\mathsf{J}_0^{-\frac{2}d}}{\mathsf{S}_d}
\]
where the last inequality is a consequence of \eqref{Eqn:DynSobolev}. Two integrations with respect to $t$ show that
\[
-\,\mathsf{H}_0\le\frac1\kappa\,\mathsf{H}_0'\,(1-e^{-\kappa\,T})\le\frac12\,\mathcal C\,\mathsf{S}_d\,\mathsf{J}_0^\frac{2}d\,\mathsf{H}_0'\quad\mbox{with}\quad\mathcal C=\frac{d+2}d\,(1-e^{-d/2})\,,
\]
which is the main result of \cite{Dolbeault2011} (when $d\ge5$), namely
\[
-\,\mathsf{H}_0\le\mathcal C\,\mathsf{S}_d\,\mathsf{J}_0^\frac{4}d\,\left[\mathsf{S}_d\,\nrm{\nabla u_0}2^2-\nrm{u_0}{2^*}^2\right]\quad\mbox{with}\quad u_0=v_0^\frac{d-2}{d+2}\,.
\]
Since this inequality holds for any initial datum $u_0=u$, we have indeed shown that
\begin{multline*}
-\,\mathsf{H}_d[v]\le\mathcal C\,\mathsf{S}_d\,\mathsf{J}_d[v]^\frac{4}d\,\left[\mathsf{S}_d\,\nrm{\nabla u}2^2-\nrm u{2^*}^2\right]\\
\forall\,u\in\D\,,\;v=u^\frac{d+2}{d-2}\,.
\end{multline*}
It is straightforward to check that our result of Theorem~\ref{Thm:SquareSobolev} is an improvement, not only because the restriction $d\ge5$ is removed, but also because the inequality holds with $\frac{d}{d+4}\le\mathcal C<1<\frac{d+2}d\,(1-e^{-d/2})$. In other words, the result of Theorem~\ref{Thm:SquareSobolev} is equivalent to
\be{Ineq:ResThm1}
-\,\mathsf{H}_0\le\frac12\,\mathcal C\,\mathsf{S}_d\,\mathsf{J}_0^\frac{2}d\,\mathsf{H}_0'\quad\mbox{with}\quad\mathcal C=\frac{d}{d+4}\,.
\ee
Up to now, we have not established yet the fact that $\mathcal C<1$. This is what we are now going to do.

\medskip Now let us reinject in the flow method described above our improved inequality of Theorem~\ref{Thm:SquareSobolev}, which can also be written as
\be{Eqn:HJ}
\mathcal C \,\mathsf{S}_d\,\mathsf{J}^\frac4d\,\left[\frac{d+2}{2\,d}\,\mathsf{S}_d\,\mathsf{J}'+\mathsf{J}^{1-\frac2d}\right]-\mathsf{H}\le0
\ee
if $v$ is still a positive solution of \eqref{Eqn:FD}. From Lemma~\ref{Lem:Concavity}, we deduce that
\[
\mathsf{H}'\le\kappa_0\,\mathsf{J}\quad\mbox{with}\quad\kappa_0:=\frac{\mathsf{H}'_0}{\mathsf{J}_0}\,.
\]
Since $t\mapsto\mathsf{J}(t)$ is monotone decreasing, there exists a function $\mathsf{Y}$ such that
\[
\mathsf{H}(t)=-\,\mathsf{Y}(\mathsf{J}(t))\quad\forall\,t\in[0,T)\,.
\]
Differentiating with respect to $t$, we find that
\[
-\,\mathsf{Y}'(\mathsf{J})\,\mathsf{J}'=\mathsf{H}'\le\kappa_0\,\mathsf{J}
\]
and, by inserting this expression in~\eqref{Eqn:HJ}, we arrive at
\[
\mathcal C \left(-\,\frac{d+2}{2\,d}\,\kappa_0\,\mathsf{S}_d^2\,\frac{\mathsf{J}^{1+\frac4d}}{\mathsf{Y}'}+\mathsf{S}_d\,\mathsf{J}^{1+\frac2d}\right)+\mathsf{Y}\le0\,.
\]
Summarizing, we end up by considering the differential inequality
\be{Ineq:EDOY}
\mathsf{Y}'\(\mathcal C\,\mathsf{S}_d\,s^{1+\frac2d}+\mathsf{Y}\)\le\frac{d+2}{2\,d}\,\mathcal C\,\kappa_0\,\mathsf{S}_d^2\,s^{1+\frac4d}\,,\quad\mathsf{Y}(0)=0\,,\quad\mathsf{Y}(\mathsf{J}_0)=-\,\mathsf{H}_0
\ee
on the interval $[0,\mathsf{J}_0]\ni s$. It is then possible to obtain estimates as follows. On the one hand we know that
\[
\mathsf{Y}'\le\frac{d+2}{2\,d}\,\kappa_0\,\mathsf{S}_d\,s^\frac2d
\]
and, hence,
\[
\mathsf{Y}(s)\le\frac12\,\kappa_0\,\mathsf{S}_d\,s^{1+\frac2d}\quad\forall\,s\in[0,\mathsf{J}_0]\,.
\]
On the other hand, after integrating by parts on the interval $[0,\mathsf{J}_0]$, we get
\[
\frac12\,\mathsf{H}_0^2-\mathcal C\,\mathsf{S}_d\,\mathsf{J}_0^{1+\frac2d}\,\mathsf{H}_0\le\frac14\,\mathcal C\,\kappa_0\,\mathsf{S}_d^2\,\mathsf{J}_0^{2+\frac4d}+\frac{d+2}d\,\mathcal C\,\mathsf{S}_d\int_0^{\mathsf{J}_0}s^\frac2d\,\mathsf{Y}(s)\;ds\,.
\]
Using the above estimate, we find that
\[
\frac{d+2}d\,\mathsf{S}_d\int_0^{\mathsf{J}_0}s^\frac2d\,\mathsf{Y}(s)\;ds\le\frac14\,\mathsf{J}_0^{2+\frac4d}\,,
\]
and finally
\[
    \frac12\,\mathsf{H}_0^2-\mathcal C\,\mathsf{S}_d\,\mathsf{J}_0^{1+\frac2d}\,\mathsf{H}_0\le\frac12\,{\mathcal C}\kappa_0\,\mathsf{S}_d^2\,\mathsf{J}_0^{2+\frac4d}\,.
\]
This is a strict improvement of \eqref{Ineq:ResThm1} when $\mathcal C=1$ since \eqref{Ineq:ResThm1} is then equivalent to
\[
-\,\mathsf{S}_d\,\mathsf{J}_0^{1+\frac2d}\,\mathsf{H}_0\le\frac12\,\mathcal C\kappa_0\,\mathsf{S}_d^2\,\mathsf{J}_0^{2+\frac4d}\,.
\]
However, it is a strict improvement of \eqref{Ineq:ResThm1} if $\mathcal C<1$ only when $|\mathsf{H}_0|=-\,\mathsf{H}_0$ is large enough (we will come back to this point in Remarks~Ê\ref{Rem10} and~\ref{Rem11}). Altogether, we have shown an improved inequality that can be stated as follows.
\begin{thm}\label{Thm:Improved} Assume that $d\ge3$. Then we have
\begin{multline*}
0\le\mathsf{H}_d[v]+\mathsf{S}_d\,\mathsf{J}_d[v]^{1+\frac{2}d}\,\varphi\(\mathsf{J}_d[v]^{\frac2d-1}\,\left[\mathsf{S}_d\,\nrm{\nabla u}2^2-\nrm u{2^*}^2\right]\)\\
\forall\,u\in\D\,,\;v=u^\frac{d+2}{d-2}
\end{multline*}
where $\varphi(x):=\sqrt{\mathcal C^2+2\,\mathcal C\,x}-\mathcal C$ for any $x\ge0$.
\end{thm}
\begin{proof} We have shown that $y^2+\,2\,\mathcal C\,y-\,\mathcal C\,\kappa_0\le0$ with $y=-\,\mathsf{H}_0/(\mathsf{S}_d\,\mathsf{J}_0^{1+\frac{2}d})\ge0$. This proves that $y\le\sqrt{\mathcal C^2+\mathcal C\kappa_0}-\mathcal C$, which proves that
\[
-\,\mathsf{H}_0\le\mathsf{S}_d\,\mathsf{J}_0^{1+\frac{2}d}\(\sqrt{\mathcal C^2+\mathcal C\,\kappa_0}-\mathcal C\)
\]
after recalling that
\[
\frac12\,\kappa_0=\frac{\mathsf{H}'_0}{\mathsf{J}_0}=\mathsf{J}_d[v_0]^{\frac2d-1}\,\left[\mathsf{S}_d\,\nrm{\nabla u_0}2^2-\nrm{u_0}{2^*}^2\right]\,.
\]\end{proof}

\begin{remark}\label{Rem10} We may observe that $x\mapsto x-\varphi(x)$ is a convex nonnegative function which is equal to $0$ if and only if $x=0$. Moreover, we have
\[
\varphi(x)\le x\quad\forall\,x\ge0
\]
with equality if and only if $x=0$. However, one can notice that 
\[
\varphi(x)\le\mathcal C\,x\quad\Longleftrightarrow\quad x\ge2\,\frac{1-\mathcal C}{\mathcal C}\,.
\]
\end{remark} 
\begin{remark}\label{Rem11} A more careful analysis of \eqref{Ineq:EDOY} shows that
\[
\mathsf{Y}(s)\le\tfrac12\(\sqrt{1+\tfrac{2\,\kappa_0}{\mathcal C}}-1\)\,\mathcal C\,\mathsf{S}_d\,s^{1+\frac2d}\,,
\]
which shows that the inequality of Theorem~\ref{Thm:Improved} holds with the improved function 
\[
\varphi(x):=\sqrt{\mathcal C^2+\mathcal C\,x+\tfrac 12\,\mathcal C^2\(\sqrt{1+\tfrac {4\,x}{\mathcal C}}-1\)}-\mathcal C
\]
but again the reader is invited to check that $\varphi(x)\le x$ for any $x\ge0$ and $\lim_{x\to0_+}\varphi(x)/x=1$.
\end{remark}

\begin{cor}\label{Cor:NonOpt} With the above notations, we have $\mathcal C<1$.\end{cor}
\begin{proof} Assume by contradiction that $\mathcal C=1$. With the notations of Section~\ref{Sec:Linearization}, let us consider a minimizing sequence $(u_n)_{n\in\N}$ for the functional $u\mapsto\frac{\mathcal F[u]}{\mathcal G[u^q]}$ but assume that $\mathsf{J}_d[u_n^q]=\mathsf{J}_d[u_\star^q]=:\mathsf{J}_\star$ for any $n\in\N$. This condition is not restrictive because of the homogeneity of the inequality. It implies that $(\mathcal G[u_n^q])_{n\in\N}$ is bounded.

If $\lim_{n\to\infty}\mathcal G[u_n^q]>0$, then we also have $\mathsf{L}:=\lim_{n\to\infty}\mathcal F[u_n]>0$, at least up to the extraction of a subsequence. As a consequence we find that
\begin{multline*}
0=\lim_{n\to\infty}\(\mathsf{S}_d\,\mathsf{J}_\star^\frac4d\,\mathcal F[u_n]-\mathcal G[u_n^q]\)\\=\mathsf{S}_d\,\lim_{n\to\infty}\left[\mathsf{J}_\star^\frac4d\,\mathcal F[u_n]-\mathsf{J}_\star^{1+\frac{2}d}\,\varphi\(\mathsf{J}_\star^{\frac{2}d-1}\,\mathcal F[u_n]\)\right]\\+\lim_{n\to\infty}\left[\mathsf{S}_d\,\mathsf{J}_\star^{1+\frac{2}d}\,\varphi\(\mathsf{J}_\star^{\frac{2}d-1}\,\mathcal F[u_n]\)-\mathcal G[u_n^q]\right]\,,
\end{multline*}
a contradiction since the last term is nonnegative by Theorem~\ref{Thm:Improved} and, as observed in Remark~\ref{Rem10}, $\mathsf{J}_\star^{4/d}\,\mathcal F[u_n]-\mathsf{J}_\star^{1+2/d}\,\varphi\big(\mathsf{J}_\star^{2/d-1}\,\mathcal F[u_n]\big)$ is positive unless $\mathcal F[u_n]=0$.

Hence we know that $\mathsf{L}=\lim_{n\to\infty}\mathcal F[u_n]=0$ and $\lim_{n\to\infty}\mathcal G[u_n^q]=0$.
According to the caracterisation of minimizers of~$\mathcal G$ by Lieb~\cite[Theorem 3.1]{MR717827}, we know that up to translations and dilations,~$u_k$
converges to~$u_\star$.
Thus there exists~$f_k$ such that~$u_k = u_\star+f_k$ with~$f_k\to 0$, and then
\[
\frac 1{\mathcal C} = \frac{\mathsf S_d}{\mathsf C_d} = \lim_{k\to\infty} \frac{1}{d^2\,(d+2)^2} \frac{\mathsf F[f_k]}{\mathsf G[f_k]} \ge \frac{d+4}{d}\,.
\]
This shows that~$\mathcal C \le \frac{d}{d+4}$, a contradiction.
\end{proof}
We may observe that $\mathcal C<1$ means $\mathsf{C}_d<\mathsf{S}_d$. This completes the  the proof of Theorem~\ref{Thm:SquareSobolev}.

\section{Caffarelli-Kohn-Nirenberg inequalities and duality}\label{Sec:CKN}

Let $2^*:=\infty$ if $d=1$ or $2$, $2^*:=2\,d/(d-2)$ if $d\ge3$ and $a_c:=(d-2)/2$. Consider the space $\mathcal D_a^{1,2}(\R^d)$ obtained by completion of $\mathcal D(\R^d\setminus\{0\})$ with respect to the norm $u\mapsto\nrm{\,|x|^{-a}\,\nabla u\,}2^2$. In this section, we shall consider the Caffarelli-Kohn-Nirenberg inequalities
\be{Ineq:CKN}
\(\;\ird{\frac{|u|^p}{|x|^{bp}}}\)^{\frac2p}\le\mathsf{C}_{a,b}\ird{\frac{|\nabla u|^2}{|x|^{2a}}}
\ee
These inequalities generalize to $\mathcal D_a^{1,2}(\R^d)$ the Sobolev inequality~\eqref{Ineq:Sobolev} and in particular the exponent $p$ is given in terms of $a$ and $b$ by
\[
p=\frac{2\,d}{d-2+2\,(b-a)}
\]
as can be checked by a simple scaling argument. A precise statements on the range of validity of~\eqref{Ineq:CKN} goes as follows.
\begin{lem}\label{Lem:CKN}{\rm \cite{MR768824}} Let $d\ge 1$. For any $p\in [2, 2^*]$ if $d\ge3$ or $p\in [2, 2^*)$ if $d=1$ or $2$, there exists a positive constant $\mathsf{C}_{a,b}$ such that \eqref{Ineq:CKN} holds if $a$, $b$ and $p$ are related by $b=a-a_c+d/p$, with the restrictions $a<a_c$, $a\le b\le a+1$ if $d\ge3$, $a<b\le a+1$ if $d=2$ and $a+1/2<b\le a+1$ if $d=1$.\end{lem}
At least for radial solutions in $\R^d$, weights can be used to work as in Section~\ref{Sec:Square} as if the dimension $d$ was replaced by the \emph{dimension} $(d-2a)$. We will apply this heuristic idea to the case $d=2$ and $a<0$, $a\to0$ in order to prove Theorem~\ref{Thm:Onofri}.  See~\ref{Sec:CKNreview} for symmetry results for optimal functions in~\eqref{Ineq:CKN}.

\medskip On $\mathcal D_a^{1,2}(\R^d)$, let us define the functionals
\[
\mathsf{F}_1[u]:=\frac{1}2\(\;\ird{\frac{|u|^p}{|x|^{bp}}}\)^{\frac2p}\quad\mbox{and}\quad\mathsf{F}_2[u]:=\frac{1}2\,\mathsf{C}_{a,b}\ird{\frac{|\nabla u|^2}{|x|^{2a}}}
\]
so that Inequality~\eqref{Ineq:CKN} amounts to $\mathsf{F}_1[u]\le\mathsf{F}_2[u]$. Assume that $\scalar\cdot\cdot$ denotes the natural scalar product on $\L^2\big(\R^d,|x|^{-2a}\,dx\big)$, that is,
\[
\scalar uv:=\ird{\frac{u\,v}{|x|^{2a}}}
\]
and denote by $\|u\|=\scalar uu^{1/2}$ the corresponding norm. Consider the operators
\begin{multline*}
\mathsf{A}_a\,u:=\nabla u\,,\quad\mathsf{A}_a^*\,w:=-\nabla\cdot w+\,2a\,\frac{x}{|x|^2}\cdot w\\
\mbox{and}\quad\mathsf{L}_a\,u:=\mathsf{A}_a^*\,\mathsf{A}_a\,u=-\,\Delta u+\,2a\,\frac{x}{|x|^2}\cdot\nabla u
\end{multline*}
defined for $u$ and $w$ respectively in $\L^2\big(\R^d,|x|^{-2a}\,dx\big)$ and $\L^2\big(\R^d,|x|^{-2a}\,dx\big)^d$. Elementary integrations by parts show that
\[
\scalar u{\mathsf{L}_a\,u}=\scalar{\mathsf{A}_a\,u}{\mathsf{A}_a\,u}=\|\mathsf{A}_a\,u\|^2=\ird{\frac{|\nabla u|^2}{|x|^{2a}}}\,.
\]
If we define the Legendre dual of $\mathsf{F}_i$ by $\mathsf{F}_i^*[v]=\sup_{u\in\mathcal D_a^{1,2}(\R^d)}\(\scalar uv-\mathsf{F}_i[u]\)$, then it is clear that we formally have the inequality $\mathsf{F}_2^*[v]\le\mathsf{F}_1^*[v]$ for any $v\in\L^q(\R^d,|x|^{-(2a-\,b)\,q}\,dx)\cap\mathsf{L}_a(\mathcal D^{1,2}_a(\R^d))$, where $q$ is H\"older's conjugate of $p$, \emph{i.e.}
\[
\frac{1}p+\frac{1}q=1\,.
\]
Using the invertibility of $\mathsf{L}_a$, we indeed observe that
\[
\mathsf{F}_2^*[v]=\scalar uv-\mathsf{F}_2[u]\quad\mbox{with}\quad v=\mathsf{C}_{a,b}\,\mathsf{L}_a\,u\;\Longleftrightarrow u=\frac{1}{\mathsf{C}_{a,b}}\,\mathsf{L}_a^{-1}\,v\,,
\]
hence proving that
\[
\mathsf{F}_2^*[v]=\frac{1}{2\,\mathsf{C}_{a,b}}\,\scalar v{\mathsf{L}_a^{-1}\,v}\,.
\]
Similarly, we get that $\mathsf{F}_1^*[v]=\scalar uv-\mathsf{F}_1[u]$ with
\be{uv}
|x|^{-\,2a}\,v=\kappa^{2-p}\,|x|^{-\,bp}\,u^{p-1}
\ee
and
\[
\kappa=\(\;\ird{\frac{|u|^p}{|x|^{bp}}}\)^\frac{1}p=\scalar uv=\(\;\ird{\frac{|v|^q}{|x|^{(2a-\,b)\,q}}}\)^\frac{1}q\,,
\]
that is
\[
\mathsf{F}_1^*[v]=\frac{1}2\(\;\ird{\frac{|v|^q}{|x|^{(2a-\,b)\,q}}}\)^{\frac2q}\,.
\]
This proves the following result.
\begin{lem}\label{Lem:DualCKN} With the above notations and under the same assumptions as in Lemma~\ref{Lem:CKN}, we have
\begin{multline*}
\frac{1}{\mathsf{C}_{a,b}}\,\scalar v{\mathsf{L}_a^{-1}\,v}\le\(\;\ird{\frac{|v|^q}{|x|^{(2a-\,b)\,q}}}\)^{\frac2q}\\
\forall\,v\in\L^q(\R^d,|x|^{-(2a-\,b)\,q}\,dx)\cap\mathsf{L}_a(\mathcal D^{1,2}_a(\R^d))\,.
\end{multline*}
\end{lem}

The next step is based on the completion of the square. Let us compute
\begin{multline*}
\|\mathsf{A}_a\,u-\lambda\,\mathsf{A}_a\,\mathsf{L}_a^{-1}\,v\|^2\\
=\|\mathsf{A}_a\,u\|^2-\,2\,\lambda\,\scalar{\mathsf{A}_a\,u}{\mathsf{A}_a\,\mathsf{L}_a^{-1}\,v}+\,\lambda^2\,\scalar{\mathsf{A}_a\,\mathsf{L}_a^{-1}\,v}{\mathsf{A}_a\,\mathsf{L}_a^{-1}\,v}\\
=\|\mathsf{A}_a\,u\|^2-\,2\,\lambda\,\scalar uv+\,\lambda^2\,\scalar v{\mathsf{L}_a^{-1}\,v}\,.
\end{multline*}
With the choice $\lambda=1/\mathsf{C}_{a,b}$ and $v$ given by \eqref{uv}, we have proved the following
\begin{thm}\label{Thm:CKNsquare} Under the assumptions of Lemma~\ref{Lem:CKN} and with the above notations, for any $u\in\mathcal D^{1,2}_a(\R^d)$ and any $v\in\L^q(\R^d,|x|^{-(2a-\,b)\,q}\,dx)\cap\mathsf{L}_a(\mathcal D^{1,2}_a(\R^d))$ we have
\begin{multline*}
0\le\(\,\ird{\frac{|v|^q}{|x|^{(2a-\,b)\,q}}}\)^{\frac2q}-\frac{1}{\mathsf{C}_{a,b}}\,\scalar v{\mathsf{L}_a^{-1}\,v}\\
\le\mathsf{C}_{a,b}\ird{\frac{|\nabla u|^2}{|x|^{2a}}}-\(\,\ird{\frac{|u|^p}{|x|^{bp}}}\)^{\frac2p}
\end{multline*}
if $u$ and $v$ are related by \eqref{uv}, if $a$, $b$ and $p$ are such that $b=a-a_c+d/p$ and verify the conditions of Lemma~\ref{Lem:CKN}, and if $q=p/(p-1)$.\end{thm}
If, instead of \eqref{uv}, we simply require that
\[
|x|^{-\,2a}\,v=|x|^{-\,bp}\,u^{p-1}\,,
\]
then the inequality becomes
\begin{multline*}
0\le\mathsf{C}_{a,b}\(\,\ird{\frac{|v|^q}{|x|^{(2a-\,b)\,q}}}\)^{\frac2q}-\scalar v{\mathsf{L}_a^{-1}\,v}\\
\le\mathsf{C}_{a,b}\(\,\ird{\frac{|u|^p}{|x|^{bp}}}\)^{\frac2p(p-2)}\left[\mathsf{C}_{a,b}\ird{\frac{|\nabla u|^2}{|x|^{2a}}}-\(\,\ird{\frac{|u|^p}{|x|^{bp}}}\)^{\frac2p}\right]
\end{multline*}
Hence Theorem~\ref{Thm:CKNsquare} generalizes Theorem~\ref{Thm:SquareSobolev}, which is recovered in the special case $a=b=0$, $d\ge3$. Because of the positivity of the l.h.s.~due to Lemma~\ref{Lem:DualCKN}, the inequality in Theorem~\ref{Thm:CKNsquare} is an improvement of the Caffarelli-Kohn-Nirenberg inequality~\eqref{Ineq:CKN}. It can also be seen as an interpolation result, namely
\begin{multline*}
2\(\;\ird{\frac{|v|^q}{|x|^{(2a-\,b)\,q}}}\)^{\frac2q}=2\(\;\ird{\frac{|u|^p}{|x|^{bp}}}\)^{\frac2p}\\\
\le\mathsf{C}_{a,b}\ird{\frac{|\nabla u|^2}{|x|^{2a}}}+\frac{1}{\mathsf{C}_{a,b}}\,\scalar v{\mathsf{L}_a^{-1}\,v}
\end{multline*}
whenever $u$ and $v$ are related by \eqref{uv}. The explicit value of $\mathsf{C}_{a,b}$ is not known unless equality in~\eqref{Ineq:CKN} is achieved by radial functions, that is when \emph{symmetry} holds. See Proposition~\ref{Prop:Symmetry} in \ref{Sec:CKNreview} for some symmetry results. Now, as in \cite{MR2437030}, we may investigate the limit $(a,b)\to(0,0)$ with $b=\alpha\,a/(1+\alpha)$ in order to investigate the Onofri limit case. A key observation is that optimality in~\eqref{Ineq:CKN} is achieved by radial functions for any $\alpha\in(-1,0)$ and $a<0$, $|a|$ small enough. In that range $\mathsf{C}_{a,b}$ is known and given by~\eqref{sharp}.

\begin{proof}[Proof of Theorem~\ref{Thm:Onofri} (continued)] Theorem~\ref{Thm:Onofri} has been established for radial functions in Section~\ref{Sec:Square}. Now we investigate the general case. We shall restrict our purpose to the case of dimension $d=2$. For any $\alpha\in(-1,0)$, let us denote by $d\mu_\alpha$ the probability measure on $\R^2$ defined by $d\mu_\alpha:=\mu_\alpha\,dx$ where
\[
\mu_\alpha:=\frac{1+\alpha}\pi\,\frac{|x|^{2\,\alpha}}{(1+|x|^{2\,(1+\alpha)})^2}\,.
\]
It has been established in~\cite{MR2437030} that
\be{MTalpha}
\log\(\int_{\R^2}e^{\,u}\;d\mu_\alpha\)-\int_{\R^2} u\;d\mu_\alpha\le\frac{1}{16\,\pi\,(1+\alpha)}\,\irdeux{|\nabla u|^2}\quad\forall\;u\in\mathcal D(\R^2)\,,
\ee
where $\mathcal D(\R^2)$ is the space of smooth functions with compact support. By density with respect to the natural norm defined by each of the inequalities, the result also holds on the corresponding Orlicz space.

We adopt the strategy of \cite[Section~2.3]{MR2437030} to pass to the limit in~\eqref{Ineq:CKN} as $(a,b)\to(0,0)$ with $b=\frac{\alpha}{\alpha + 1}\,a$. Let
\[
a_\varepsilon=-\frac{\varepsilon}{1-\varepsilon}\,(\alpha+1)\,,\quad b_\varepsilon=a_\varepsilon+\varepsilon ,\quad p_\varepsilon=\frac{2}\varepsilon\,,
\]
and
\[
    u_\varepsilon (x) = \left(1+|x|^{2\,(\alpha+1)}\right)^{-\frac\varepsilon{1-\varepsilon}}\,,
\]
assuming that $u_\varepsilon$ is an optimal function for \eqref{Ineq:CKN}, define
\[
    \kappa_\varepsilon =\int_{\R^2} \left[ \frac{u_\varepsilon}{|x|^{a_\varepsilon + \varepsilon}} \right]^{2/\varepsilon}\, dx = \int_{\R^2}\frac{|x|^{2\,\alpha}}{\big(1+|x|^{2\,(1+\alpha)}\big)^2}\,\frac{u_\varepsilon^2}{|x|^{2a_\varepsilon}}\,dx =\frac\pi{\alpha+1}\,\frac{\Gamma\big(\frac1{1-\varepsilon}\big)^2}{\Gamma\big(\frac2{1-\varepsilon}\big)}\,,
\]
\[
    \lambda_\varepsilon = \int_{\R^2}\left[ \frac{|\nabla u_\varepsilon|}{|x|^a} \right]^2\,dx = 4\,a_\varepsilon^2\int_{\R^2}\frac{|x|^{2\,(2\,\alpha+1-a_\varepsilon)}}{\big(1+|x|^{2\,(1+\alpha)}\big)^{\frac2{1-\varepsilon}}}\,dx= 4\,\pi\,\frac{|a_\varepsilon|}{1-\varepsilon}\,\frac{\Gamma\big(\frac{1}{1-\varepsilon}\big)^2}{\Gamma\big(\frac{2}{1-\varepsilon}\big)}\,.
\]
Then $w_\varepsilon=(1+\frac{1}2\,\varepsilon\,u)\,u_\varepsilon$ is such that
\begin{eqnarray*}
&&\lim_{\varepsilon\to0_+}\frac{1}{\kappa_\varepsilon}\int_{\R^2}\frac{|w_\varepsilon|^{p_\varepsilon}}{|x|^{b_\varepsilon{p_\varepsilon}}}\,dx=\int_{\R^2}e^u\,d\mu_\alpha\,,\\
&&\lim_{\varepsilon\to0_+}\frac{1}\varepsilon\left[\frac{1}{\lambda_\varepsilon}\int_{\R^2}\frac{|\nabla w_\varepsilon|^2}{|x|^{2a_\varepsilon}}\,dx-1\right]=\int_{\R^2}u\,d\mu_\alpha+\frac1{16\,(1+\alpha)\,\pi}\,\nrmdeux{\nabla u}2^2\,.
\end{eqnarray*}
Hence we can recover \eqref{MTalpha} by passing to the limit in~\eqref{Ineq:CKN} as $\varepsilon\to0_+$. On the other hand, if we pass to the limit in the inequality stated in Theorem~\ref{Thm:CKNsquare}, we arrive at the following result, for any $\alpha\in(-1,0)$.
\begin{thm}\label{Thm:Onofrisquare} Let $\alpha\in(-1,0]$. With the above notations, we have
\begin{multline*}
0\le\irdeux{\kern-3pt v\,\log\(\frac{v}{\mu_\alpha}\)}-4\,\pi\,(1+\alpha)\irdeux{\kern-3pt (v-\mu_\alpha)\,(-\Delta)^{-1}\,(v-\mu_\alpha)}\\
\le\frac1{16\,\pi\,(1+\alpha)}\irdeux{|\nabla u|^2}-\log\(\int_{\R^2}e^{\,u}\;d\mu_\alpha\)+\int_{\R^2} u\;d\mu_\alpha
\end{multline*}
for any $u\in\mathcal D$, where $u$ and $v$ are related by
\[
v=\frac{e^u\,\mu_\alpha}{\irdmua{e^u}}\,.
\]
\end{thm}
The case $\alpha=0$ is achieved by taking the limit as $\alpha\to0_-$. Since $-\Delta\log \mu_\alpha=8\,\pi\,(1+\alpha)\,\mu_\alpha$ holds for any $\alpha\in(-1,0]$, the proof of Theorem~\ref{Thm:Onofri} is now completed, with $\mu=\mu_0$.\end{proof}

\appendix\section{Some useful formulae}\label{Sec:UsefulFormulae}

We recall that
\[
f(q):=\int_{\R}\frac{dt}{(\cosh t)^q}=\frac{\sqrt\pi\;\Gamma(\frac{q}2)}{\Gamma(\frac{q+1}2)}
\]
for any $q>0$. An integration by parts shows that $f(q+2)=\frac{q}{q+1}\,f(q)$. The following formulae are reproduced with no change from~\cite{DDFT} (also see~\cite{DEL2011,DE2013}). The function $w(t):=(\cosh t)^{-\frac2{p-2}}$ solves
\[
-(p-2)^2\,w''+4\,w-2\,p\,w^{p-1}=0
\]
and we can define
\[
\mathsf{I}_q:=\irt{|w(t)|^q}\quad\mbox{and}\quad \mathsf{J}_2:=\irt{|w'(t)|^2}\,.
\]
Using the function $f$, we can compute $\mathsf{I}_2=f\big(\frac{4}{p-2}\big)$, $\mathsf{I}_p=f\big(\frac{2\,p}{p-2}\big)=f\big(\frac{4}{p-2}+2\big)$ and get the relations
\[
\mathsf{I}_2=\frac{\sqrt{\pi}\;\Gamma\big(\frac2{p-2}\big)}{\Gamma\big(\frac{p+2}{2\,(p-2)}\big)}\,,\quad \mathsf{I}_p=\frac{4\,\mathsf{I}_2}{p+2}=\frac{4\,\sqrt\pi\,\Gamma\(\frac2{p-2}\)}{(p+2)\,\Gamma\(\frac{p+2}{2\,(p-2)}\)}\,,\quad \mathsf{J}_2=\frac{4\,\mathsf{I}_2}{(p+2)\,(p-2)}\,.
\]
In particular, this establishes~\eqref{Eqn:sd}, namely
\[
\mathsf{s}_d=\frac{\mathsf{I}_p^{1-\frac{2}d}}{\mathsf{J}_2+\frac14\,(d-2)^2\,\mathsf{I}_2}\,,\quad\mbox{with}\;p=\frac{2\,d}{d-2}
\]
for any $d>2$. The expression of the optimal constant in Sobolev's inequality~\eqref{Ineq:Sobolev}: $\mathsf{S}_d=\mathsf{s}_d\,|\S^{d-1}|^{-2/d}$, where
\[
|\S^{d-1}|=\frac{2\,\pi^{d/2}}{\Gamma(d/2)}
\]
denotes the volume of the unit sphere, for any integer $d\ge3$, follows from the duplication formula
\[
2^{d-1}\,\Gamma\(\tfrac d2\)\,\Gamma\(\tfrac{d+1}2\)=\sqrt\pi\;\Gamma(d)
\]
according for instance to \cite{MR0167642}. See \cite[Appendix B.4]{DoEsLa} for further details.

\section{Poincar\'e inequality and stereographic projection}\label{Sec:Stereographic}

On $\S^d\subset\R^{d+1}$, consider the coordinates $\omega=(\rho\,\phi,z)\in\R^d\times\R$ such that $\rho^2+z^2=1$, $z\in[-1,1]$, $\rho\ge0$ and $\phi\in\S^{d-1}$, and define the \emph{stereographic projection} $\Sigma:\S^d\setminus\{\mathrm N\}\to\R^d$ by $\Sigma(\omega)=x=r\,\phi$ and
\[
z=\frac{{}r^2-1}{r^2+1}=1-\frac{2}{r^2+1}\;,\quad \rho=\frac{{}2\,r}{r^2+1}\,.
\]
The \emph{North Pole} $\mathrm N$ corresponds to $z=1$ (and is formally sent at infinity) while the \emph{equator} (corresponding to $z=0$) is sent onto the unit sphere $\S^{d-1}\subset\R^d$. Now we can transform any function $v$ on $\S^d$ into a function $u$ on $\R^d$ using
\[
v(\omega)=\big(\tfrac r\rho\big)^\frac{d-2}2\,u(x)=\big(\tfrac{r^2+1}2\big)^\frac{d-2}2\,u(x)=(1-z)^{-\frac{d-2}2}\,u(x)\,.
\]
A standard computation shows that
\[
\int_{\S^d}|\nabla v|^2\;d\omega+\frac{1}4\,d\,(d-2)\int_{\S^d}|v|^2\;d\omega=\ird{|\nabla u|^2}
\]
and
\[\int_{\S^d}|v|^q\;d\omega=\ird{|u|^q\,\big(\tfrac2{1+|x|^2}\big)^{d-(d-2)\frac{q}2}}\,.
\]
On $\S^d$, the kernel of the Laplace-Beltrami operator is generated by the constants and the lowest positive eigenvalue is $\lambda_1=d$. The corresponding eigenspace is generated by $v_0(\omega)=1$ and $v_i(\omega)=\omega_i$, $i=1$, $2$, \ldots $d+1$. All eigenvalues of the Laplace-Beltrami operator are given by the formula
\[
\lambda_k=k\,(k+d-1)\quad\forall\,k\in\N
\]
according to \cite{MR0282313}. We still denote by $u_\star$ the Aubin-Talenti extremal function
\[
u_\star(x):=(1+|x|^2)^{-\frac{d-2}2}\quad\forall\,x\in\R^d\,.
\]
Using the inverse stereographic projection, the reader is invited to check that Sobolev's inequality is equivalent to the inequality
\[
\frac{4}{d\,(d-2)}\int_{\S^d}|\nabla v|^2\;d\omega+\int_{\S^d}|v|^2\;d\omega\ge|\S^d|^\frac2d\,\(\int_{\S^d}|v|^\frac{2\,d}{d-2}\;d\omega\)^\frac{d-2}d
\]
so that the Aubin-Talenti extremal function is transformed into a constant function on the sphere and incidentally this shows that
\[
\mathsf{S}_d=\frac{4}{d\,(d-2)}\,|\S^d|^{-\frac2d}\,.
\]

With these preliminaries on the Laplace-Beltrami operator and the stereographic projection in hand, we can now state the counterpart on $\R^d$ of the Poincar\'e inequality on $\S^d$.
\begin{lem}\label{Lem:Poincare} For any function $f\in\mathcal D^{1,2}(\R^d)$ such that
\begin{multline*}
\ird{f\,\frac{u_\star}{(1+|x|^2)^2}}=0\,,\quad\ird{f\,\frac{(1-|x|^2)\,u_\star}{(1+|x|^2)^3}}=0\,,\\
\mbox{and}\quad\ird{f\,\frac{x_i\,u_\star}{(1+|x|^2)^3}}=0\quad\forall\,i=1\,,\;2\,,\;\ldots d
\end{multline*}
the following inequality holds
\[
\ird{|\nabla f|^2}\ge(d+2)\,(d+4)\ird{\frac{f^2}{(1+|x|^2)^2}}\,.
\]
\end{lem}
\begin{proof} On the sphere we know that
\begin{align*}
\int_{\S^d}|\nabla v|^2\;d\omega+\frac{1}4\,d\,(d-2)\int_{\S^d}v^2\;d\omega&\ge\(\lambda_2+\frac14\,d\,(d-2)\)\int_{\S^d}v^2\;d\omega\\
&=\frac14\,(d+2)(d+4) \int_{\S^d}v^2\;d\omega
\end{align*}
if $v$ is orthogonal to $v_i$ for any $i=0$, $1$, \ldots $d+1$. The conclusion follows from the stereographic projection.\end{proof}

\section{Flow on the sphere and consequences}\label{Sec:FlowSphere}

We recall that Equation~\eqref{Eqn:FD} admits special solutions with \emph{separation of variables} given by
\be{Eqn:Separation}
v_\star(t,x)=\lambda^{(d+2)/2}\,(T-t)^{\frac{d+2}{4}}\,(u_\star((x-x_0)/\lambda))^\frac{d+2}{d-2}
\ee
where $u_\star(x):=(1+|x|^2)^{-(d-2)/2}$ is the Aubin-Talenti extremal function, $x\in\R^d$ and $0<t<T$. Such a solution is generic near the extinction time~$T$, in the following sense.
\begin{lem}\label{Thm:delPinoSaez}{\rm \cite{delPino-Saez01,MR2282669}}. For any solution $v$ of~\eqref{Eqn:FD} with nonnegative, not identically zero initial datum $v_0\in\L^{2d/(d+2)}(\R^d)$, there exists $T>0$, $\lambda>0$, $c>0$ and
$x_0\in\R^d$ such that $v(t,\cdot)\not\equiv0$ for any $t\in(0,T)$ and
\[
\lim_{t\to T_-}(T-t)^{-\frac{d+2}4}\,\sup_{x\in\R^d}(1+|x|^2)^\frac{d+2}2\,\left|\,\frac{v(t,x)}{v_\star(t,x)}-c\,\right|=0
\]
if $v_\star$ is defined by~\eqref{Eqn:Separation}.\end{lem}
If $v$ solves the \emph{fast diffusion} equation~\eqref{Eqn:FD} on $\R^d$, then we may use the inverse stereographic projection (see~\ref{Sec:Stereographic}) to define the function $w$ on $\S^d$ such that
\[
v(t,x)=e^{-\frac{d+2}4\,\tau}\,\(\tfrac2{1+r^2}\)^\frac{d+2}2\,w(\tau,y)
\]
where $\tau=-\log(T-t)$, $r=|x|$ and $y=\(\tfrac{2\,x}{1+r^2},\tfrac{1-r^2}{1+r^2}\)\in\S^d\subset\R^d\times\R$. 

With no loss of generality, assume that $c=\lambda=1$ and $x_0=0$. According to Lemma~\ref{Thm:delPinoSaez}, $w$ uniformly converges as $\tau\to\infty$  to $1$ on $\S^d$. Let $d\sigma_d$ denote the measure induced on $\S^d\subset\R^{d+1}$ by Lebesgue's measure on $\R^{d+1}$. We may then write
\[
\mathsf{J}(t)=e^{-\frac{d}2\tau}\isd{w^\frac{2\,d}{d+2}}
\]
and
\[
\ird{|\nabla u^\frac{d-2}{d+2}|^2}=e^{-\frac{d-2}2\tau}\(\isd{\big|\nabla w^\frac{d-2}{d+2}\big|^2}+\frac{1}4\,d\,(d-2)\isd{\big|w^\frac{d-2}{d+2}\big|^2}\)
\]
with $\tau=-\log(T-t)$, so that $\frac{d\tau}{dt}=e^\tau$. Hence $w$ solves
\[
w_\tau=\mathcal L\,w^\frac{d-2}{d+2}-\frac{1}4\,d\,(d-2)\,w^\frac{d-2}{d+2}+\frac14\,(d+2)\,w
\]
where $\mathcal L$ denotes the Laplace-Beltrami operator on the sphere $\S^d$, and
\[
\frac{d}{dt}\mathsf{J}=-\,\frac{2\,d}{d+2}\,e^{-\frac{d-2}2\,\tau}\(\isd{\big|\nabla w^\frac{d-2}{d+2}\big|^2}+\frac{1}4\,d\,(d-2)\isd{\big|w^\frac{d-2}{d+2}\big|^2}\)\,,
\]
\[
\frac{d}{dt}\ird{|\nabla u^\frac{d-2}{d+2}|^2}=-\,2\,\frac{d-2}{d+2}\isd{\(\mathcal L\,w^\frac{d-2}{d+2}-\frac{1}4\,d\,(d-2)\,w^\frac{d-2}{d+2}\)^2\!w^{-\frac4{d+2}}}\,.
\]
Using the Cauchy-Schwarz inequality, that is, by writing that
\begin{multline*}
\left[\,\isd{\big|\nabla w^\frac{d-2}{d+2}\big|^2}+\frac{1}4\,d\,(d-2)\isd{\big|w^\frac{d-2}{d+2}\big|^2}\,\right]^2\\
=\left[\,\isd{\(\mathcal L\,w^\frac{d-2}{d+2}-\frac{1}4\,d\,(d-2)\,w^\frac{d-2}{d+2}\)w^{-\frac2{d+2}}\,w^\frac{d}{d+2}}\,\right]^2\\
\le\isd{\(\mathcal L\,w^\frac{d-2}{d+2}-\frac{1}4\,d\,(d-2)\,w^\frac{d-2}{d+2}\)^2w^{-\frac4{d+2}}}\,\isd{w^\frac{2\,d}{d+2}}\,,
\end{multline*}
we conclude that
\[
\mathsf{Q}=\mathsf{J}^{\frac2d-1}\,\isd{\big|\nabla w^\frac{d-2}{d+2}\big|^2}
\]
is monotone decreasing, and hence
\[
\mathsf{H}''=\frac{\mathsf{J}'}{\mathsf{J}}\,\mathsf{H}'+\,2\,\mathsf{J}\,\Sd\,\mathsf{Q}'\le\frac{\mathsf{J}'}{\mathsf{J}}\,\mathsf{H}'\,.
\]
This establishes the proof of Lemma~\ref{Lem:Concavity} for any $d\ge3$. 

\section{Symmetry in Caffarelli-Kohn-Nirenberg inequalities}\label{Sec:CKNreview}

In this Appendix, we recall some known results concerning \emph{symmetry} and \emph{symmetry breaking} in the Caffarelli-Kohn-Nirenberg inequalities~\eqref{Ineq:CKN}.
\begin{prop}\label{Prop:Symmetry} Assume that $d\ge2$. There exists a continuous function $\alpha:(2,2^*)\to(-\infty,0)$ such that $\lim_{p\to2^*}\alpha(p)=0$ for which the equality case in~\eqref{Ineq:CKN} is not 	achieved among radial functions if $a<\alpha(p)$ while for $a<\alpha(p)$ equality is achieved by
\[
u_\star(x):=\(1+|x|^{\frac2\delta\,(a_c-a)}\)^{-\delta}\quad\forall\,x\in\R^d
\]
where $\delta=\frac{a_c+b-a}{1+a-b}$. Moreover the function $\alpha$ has the following properties
\begin{enumerate}
\item[(i)] For any $p\in(2,2^*)$, $\alpha(p)\ge a_c-2\,\sqrt{\frac{d-1}{p^2-4}}$.
\item[(ii)] For any $p\in\big(2,2\,\frac{d^2-d+1}{d^2-3\,d+3}\big)$, $\alpha(p)\le a_c-\frac12\,\sqrt{\frac{(d-1)\,(6-p)}{p-2}}$.
\item[(iii)] If $d=2$, $\lim_{p\to2^*}\beta(p)/\alpha(p)=0$ where $\beta(p):=\alpha(p)-a_c+d/p$.
\end{enumerate}
\end{prop}
This result summarizes a list of partial results that have been obtained in various papers. Existence of optimal functions has been dealt with in \cite{Catrina-Wang-01}, while Condition (i) in Proposition~\ref{Prop:Symmetry} has been established in \cite{Felli-Schneider-03}. See \cite{DELT09} for the existence of the curve $p\mapsto\alpha(p)$, \cite{springerlink:10.1007/s00526-011-0394-y,Oslo} for various results on symmetry in a larger class of inequalities, and \cite{DEL2011} for Property (ii) in Proposition~\ref{Prop:Symmetry}. Numerical computations of the branches of non-radial optimal functions and formal asymptotic expansions at the bifurcation point have been collected in \cite{Freefem,1304}. The paper \cite{MR2437030} deals with the special case of dimension $d=2$ and contains Property (iii) in Proposition~\ref{Prop:Symmetry}, which can be rephrased as follows: the region of radial symmetry contains the region corresponding to $a\ge\alpha(p)$ and $b\ge\beta(p)$, and the parametric curve $p\mapsto(\alpha(p),\beta(p))$ converges to $0$ as $p\to2^*=\infty$ tangentially to the axis $b=0$. For completeness, let us mention that  \cite[Theorem~3.1]{MR1734159} covers the case $a>a_c-d/p$ also we will not use it. Finally, let us observe that in the \emph{symmetric case}, the expression of $\mathsf{C}_{a,b}$ can be computed explicitly in terms of the $\Gamma$ function as
\be{sharp}
\mathsf{C}_{a,b}= |\S^{d-1}|^{\frac{p-2}p} {\textstyle\left[\frac{(a-a_c)^2\,(p-2)^2}{p+2}\right]^\frac{p-2}{2\,p}\left[\frac{p+2}{2\,p\,(a-a_c)^2}\right]\left[\frac{4}{p+2}\right]^\frac{6-p}{2\,p}} \left[\tfrac{\Gamma\(\frac2{p-2}+\frac{1}2\)}{\sqrt\pi\;\Gamma\(\frac2{p-2}\)}\right]^\frac{p-2}p
\ee
where the volume of the unit sphere is given by $|\S^{d-1}|=2\,\pi^\frac{d}2/\,\Gamma\(\frac{d}2\)$.

\par\medskip\centerline{\rule{2cm}{0.2mm}}\medskip
\begin{spacing}{0.9}
\noindent{\small{\bf Acknowlegments.} This work has been partially supported by the projects \emph{STAB}, \emph{NoNAP} and \emph{Kibord} of the French National Research Agency (ANR). The authors warmfully thank Mr.~Nguyen Van Hoang, who found a mistake in a previous version of this paper. \par\medskip\noindent{\small\copyright\,2014 by the authors. This paper may be reproduced, in its entirety, for non-commercial purposes.}}\end{spacing}

\end{document}